\documentclass[12pt, a4paper]{amsart}

\usepackage[english]{babel}
\usepackage{amsmath}
\usepackage{amssymb}
\usepackage{amscd} 
\usepackage{hyperref} 
\usepackage{verbatim} 
\usepackage{color}  
\usepackage{tikz-cd}\usetikzlibrary{babel} 
\usepackage{enumerate} 
\usepackage{mathtools} 
\usepackage{cite}
\usepackage[initials,msc-links,backrefs]{amsrefs}
\usepackage[all]{xy}

\title[Profiniteness of Volume]
 {Profiniteness of higher rank volume}

 \author[H. Kammeyer]{Holger Kammeyer}
 \author[S. Kionke]{Steffen Kionke}
 \author[R. K\"ohl]{Ralf K\"ohl}
 
  \address{Heinrich Heine University D{\"u}sseldorf, Faculty of Mathematics and Natural Sciences, Mathematical Institute, Germany}
 \email{holger.kammeyer@hhu.de}

 \address{FernUniversit\"at in Hagen, Faculty of Mathematics and Computer Science, Germany}
 \email{steffen.kionke@fernuni-hagen.de}

 \address{Kiel University, Faculty of Mathematics and Natural Sciences, Department of Mathematics, Germany}
 \email{koehl@math.uni-kiel.de}

\subjclass[2010]{22E40, 20E18}
\keywords{profinite rigidity, volume}

\theoremstyle{plain}
\newtheorem{theorem}{Theorem}
\newtheorem{lemma}[theorem]{Lemma}
\newtheorem{corollary}[theorem]{Corollary}
\newtheorem{proposition}[theorem]{Proposition}
\newtheorem{conjecture}[theorem]{Conjecture}
\newtheorem*{conjectureU}{Conjecture (U)}

\theoremstyle{definition}
\newtheorem{definition}[theorem]{Definition}
\newtheorem{remark}[theorem]{Remark}
\newtheorem{example}[theorem]{Example}
\newtheorem{question}[theorem]{Question}

\numberwithin{equation}{section}
\numberwithin{theorem}{section}

\DeclareMathOperator{\Hom}{Hom}

\DeclareMathOperator{\rank}{rk}

\providecommand{\normal}{\trianglelefteq}

\providecommand{\fg}{\mathfrak{g}}

\providecommand{\bbR}{\mathbb{R}}
\providecommand{\bbQ}{\mathbb{Q}}
\providecommand{\bbZ}{\mathbb{Z}}

\providecommand{\bbA}{\mathbb{A}}
\providecommand{\bbC}{\mathbb{C}}

\DeclareMathOperator{\vol}{vol}

\renewcommand{\phi}{\varphi}

\providecommand{\ignore}[1]{}
\providecommand{\alg}[1]{\mathbf{#1}}

\providecommand{\R}{\mathbb{R}}
\providecommand{\Q}{\mathbb{Q}}
\providecommand{\Z}{\mathbb{Z}}

\providecommand{\C}{\mathbb{C}}

\newcounter{commentcounter}

\usepackage{ifthen,srcltx}
\newcommand{\showcomments}{yes}

\newsavebox{\commentbox}
\newenvironment{comnz}%
{\ifthenelse{\equal{\showcomments}{yes}}%
{\footnotemark
        \begin{lrbox}{\commentbox}
        \begin{minipage}[t]{1.25in}\raggedright\sffamily\tiny
        \footnotemark[\arabic{footnote}]}
{\begin{lrbox}{\commentbox}}}
{\ifthenelse{\equal{\showcomments}{yes}}
{\end{minipage}\end{lrbox}\marginpar{\usebox{\commentbox}}}
{\end{lrbox}}}

\begin{document}
\selectlanguage{english}

\begin{abstract}
  We show that the covolume of an irreducible lattice in a higher rank semisimple Lie group with the congruence subgroup property is determined by the profinite completion.  Without relying on CSP, we additionally show that volume is a profinite invariant of octonionic hyperbolic congruence manifolds.
\end{abstract}

\maketitle

\section{Introduction}

A group \(\Gamma\) is called \emph{residually finite} if every \(1 \neq g \in \Gamma\) maps to a nontrivial element in some finite quotient of \(\Gamma\).  It then becomes a natural question in how far \(\Gamma\), or at least some property of \(\Gamma\), is determined by all finite quotient groups of \(\Gamma\); or equivalently, by the profinite completion \(\widehat{\Gamma}\).  A well-known problem of this sort has been around for quite some time and was in particular advertised as the final problem in A.\,Reid's 2018 ICM address~\cite{Reid:ICM}*{Question~7.4}.

\begin{question} \label{question:volume-3-manifolds}
Let \(M\) and \(N\) be finite volume hyperbolic 3-manifolds.  Suppose that \(\widehat{\pi_1 M} \cong \widehat{\pi_1 N}\).  Does \(\operatorname{vol} M = \operatorname{vol} N\)?
\end{question}

This question is in fact the case \(G = \operatorname{SL}_2 (\C)\) of the more general question if profinitely isomorphic irreducible lattices \(\Gamma, \Delta \le G\) in a semisimple Lie group \(G\) have fundamental domains of equal Haar measure.  We answer the general question affirmatively if \(G\) has \emph{higher rank} (at least two) and possesses the \emph{congruence subgroup property}.

\begin{theorem} \label{thm:same-lie-group}
  Let \(G\) be a connected semisimple Lie group with higher rank and finite center and without compact factors.  Fix a Haar measure \(\mu\) on \(G\) and suppose \(\Gamma, \Delta \le G\) are irreducible lattices with CSP* such that \(\widehat{\Gamma} \cong \widehat{\Delta}\).  Then \(\mu(G/\Gamma) = \mu(G/\Delta)\).
\end{theorem}

Conjecturally, the assumption of CSP* is always satisfied.  Nevertheless, we highlight the following unconditional conclusion for lattices with non-compact quotient.

\begin{theorem} \label{thm:non-uniform-same-lie-group}
    Let \(G\) be a connected semisimple Lie group with higher rank and finite center and without compact factors.  Fix a Haar measure \(\mu\) on \(G\) and let \(\Gamma, \Delta \le G\) be irreducible non-uniform lattices such that \(\widehat{\Gamma} \cong \widehat{\Delta}\).  Then \(\mu(G/\Gamma) = \mu(G/\Delta)\).
\end{theorem}

We will exlain the precise meaning of CSP* in Section~\ref{sec:csp}.  At this point, let us only inform the experts that essentially, it shall refer to a finite \emph{congruence kernel} with the two additional requirements that ``Conjecture U'' should hold true if the algebraic group in which $\Gamma$ is arithmetic is a certain outer form of type $A_n$ and moreover that the congruence kernel has the same order as the \emph{metaplectic kernel}.  The well-known \emph{Margulis--Platonov conjecture} asserts that the latter condition should actually be automatic and this is in fact known in the majority of cases.  Even better, another well-known conjecture due to Serre says that CSP should hold under our assumptions and the status of this conjecture is likewise advanced.  So in many cases, requiring CSP* is not needed and conjecturally, it is redundant altogether.  A notable case in which CSP* is known occurs if the algebraic group is isotropic.  By the Borel--Harish-Chandra Theorem, this translates back to Theorem~\ref{thm:non-uniform-same-lie-group}.

\medskip
While our methods break down with regard to the original Question~\ref{question:volume-3-manifolds}, we do offer a result for another type of rank one locally symmetric spaces, for which the congruence subgroup property is still unknown.  Recall that there exists an exceptional rank one symmetric space called the \emph{octonionic hyperbolic plane} \(\mathbb{OH}^2\).  We will refer to any 16-dimensional connected Riemannian manifold whose universal covering is isometric to \(\mathbb{OH}^2\) as an \emph{octonionic hyperbolic manifold}.  Note that the isometry group of \(\mathbb{OH}^2\) is the exceptional real Lie group \(F_{4(-20)}\) so that the fundamental group of any octonionic hyperbolic manifold embeds as a subgroup of \(F_{4(-20)}\) by the deck transformation action.

\begin{theorem} \label{thm:octonionic}
  Let \(M\) and \(N\) be finite volume octonionic hyperbolic manifolds with fundamental groups \(\Gamma = \pi_1 M\) and \(\Delta = \pi_1 N\).  Assume that \(\Gamma\) and \(\Delta\) are arithmetic congruence lattices in \(F_{4(-20)}\) and \(\widehat{\Gamma} \cong \widehat{\Delta}\).  Then \(\vol M = \vol N\).
\end{theorem}

By the work of Corlette~\cite{Corlette:archimedean-superrigidity} and Gromov--Schoen~\cite{GromovSchoen}, arithmeticity for lattices in $F_{4(-20)}$ is known.  Contrary to Serre's original conjecture mentioned above, several results have meanwhile pointed in the direction that CSP might also hold in type $F_{4(-20)}$.  So if that is true, our result gives profiniteness of octonionic hyperbolic volume in general.  Moreover, in that case one can construct non-isomorphic, profinitely isomorphic lattices as in the theorem using non-isomorphic number fields with isomorphic adele rings.  This contrasts with Question~\ref{question:volume-3-manifolds}:  Conjecturally, all Kleinian groups are profinitely rigid.

\medskip
Finally, it has long been known that even lattices \(\Gamma \le G\) and \(\Delta \le H\) in different Lie groups can have isomorphic profinite completions (e.g.~\cite{Aka:arithmetic}). One may still ask if they should have equal covolume.  This question, however, depends on the normalization of the Haar measures on \(G\) and \(H\).  Two more or less canonical such normalizations come to mind.  One can use the \emph{Killing form} to obtain such a normalization as we will explain in Section~\ref{sec:measures}.  Another option is to use the \emph{Euler-Poincar{\'e}-measure} \cite{Serre:cohomologie}*{Section~1.6} which is the Haar measure \(\mu\) on \(G\) such that \(\mu(G/\Gamma) = \chi (\Gamma)\) is the Euler characteristic for every torsion-free cocompact lattice \(\Gamma \le G\), provided \(\chi(\Gamma) \neq 0\).

However, neither normalization turns volume into a profinite invariant for lattices in higher rank Lie groups.  While it will be apparent from this investigation that the Killing form normalization does not work, it was shown before in \cite{Kammeyer-Kionke-Raimbault-Sauer}*{Theorem~1.2} that there exist profinitely isomorphic spinor groups with distinct Euler characteristics.

So the only thing one can still hope for is that on each higher rank Lie group, one can fix one particular normalization of the Haar measure such that volume becomes a profinite invariant among all lattices in all such Lie groups.  This is indeed what will be accomplished in this article, even in the more general setting where a ``semisimple Lie group'' is understood as a product of simple Lie groups over various local fields of characteristic zero.  To formulate these most general results precisely, we shall now delve a bit deeper into the theory.

\begin{definition} \label{def:simply-connected}
Let $A$ be a finite set. For each $\alpha \in A$ let $k_\alpha$ be a local field with $\mathrm{char}(k_\alpha)=0$ and let $\mathbf{G}_\alpha$ be a simply-connected absolutely almost simple linear algebraic $k_\alpha$-group.  We call a locally compact topological group $G$ of the form
\[
	G = \prod_{\alpha\in A} \mathbf{G}_\alpha(k_\alpha)
\]
an \emph{algebraically simply-connected semisimple Lie group}.
The \emph{rank} of $G$ is defined as
\[ \rank G = \sum_{\alpha \in A} \rank_{k_\alpha} \mathbf{G}_\alpha. \]
We say that $G$ \emph{has no compact factors} if each $\mathbf{G}_\alpha$ is $k_\alpha$-isotropic or, equivalently, if none of the groups $\mathbf{G}_\alpha(k_\alpha)$ is compact.
\end{definition}

The point of this definition is that an irreducible lattice $\Gamma \le G$ in an algebraically simply-connected semisimple Lie group with $\rank G \ge 2$ and without compact factors is \emph{$S$-arithmetic} by \emph{Margulis arithmeticity} \cite{Margulis:discrete-subgroups}*{Theorem~IX.1.11, p.\,298}.  This means that there exists a number field $k$, a connected absolutely almost simple $k$-group $\mathbf{H}$, and a finite subset $S \subset V(k)$ of the set of all places of $k$ such that $S$ contains all infinite places and finally, there exists a continuous homomorphism of topological groups
\[ \varphi \colon \prod_{v \in S^{\text{is}}} \mathbf{H}({k_v}) \longrightarrow G \]
such that $\varphi(\mathbf{H}(\mathcal{O}_{k, S}))$ is commensurable with $\Gamma$.  Here, $S^{\text{is}}$ denotes the subset of $S$ with all infinite places removed at which $\mathbf{H}$ is anisotropic and $\mathcal{O}_{k, S}$ denotes the ring of $S$-integers in $k$, meaning the subring of $k$ consisting of all $x \in k$ with $|x|_v \le 1$ for all finite places $v \notin S$.  The group $\mathbf{H}(\mathcal{O}_{k, S})$ is defined by picking an embedding $\mathbf{H} \subset \mathbf{GL_n}$ and is thus well-defined up to commensurability.  By~\cite{Margulis:discrete-subgroups}*{Remark~IX.1.6.(i), p.\,294}, we may and will moreover assume that $\mathbf{H}$ is simply-connected.  Moreover, our assumption that the groups $\mathbf{G}_\alpha$ are all absolutely simple and simply-connected guarantees that $\varphi$ is in fact an isomorphism of topological groups given by a product of isomorphisms of the factors~\cite{Margulis:discrete-subgroups}*{Remarks~(i) and~(iii) on p.\,291}.  In fact, we have a bijection $\alpha \colon S^{\text{is}} \rightarrow A$ such that $\varphi = \prod_{v \in S^{\text{is}}} \varphi_v$ for isomorphisms $\varphi_v \colon \mathbf{H}_v \rightarrow \mathbf{G}_{\alpha(v)}$ defined over isomorphism $k_v \cong k_{\alpha(v)}$.

\medskip
Finally, it follows from superrigidity~\cite{Margulis:discrete-subgroups}*{Theorem~C, Chapter~VIII, p.\,259} that $k$, $\mathbf{H}$, and $S$ are unique in the strongest sense:  For $k'$, $\mathbf{H}'$, and $S'$ with the above properties, there exists a field isomorphism $\sigma \colon k \rightarrow k'$ inducing a bijection from $S$ to $S'$ and there exists a $k$-isomorphism $\mathbf{H} \cong {}^\sigma \mathbf{H}'$.  In particular, it is meaningful to define that $\Gamma$ has the \emph{congruence subgroup property (CSP)} if the uniquely defined $k$-group $\mathbf{H}$ has finite $S$-congruence kernel $C(\mathbf{H}, S)$.  We remark that we could have equivalently and intrinsically defined that $\Gamma$ has CSP by requiring that it have polynomial representation growth~\cite{Lubotzky-Martin:rep-growth}, or, still equivalently, polynomial index growth~\cite{Lubotzky-Segal:subgroup-growth}*{Theorem~12.10, p.\,223}.  Similarly, we define that \(\Gamma\) has CSP* if \(\mathbf{H}\) has CSP* with respect to \(S\) according to Definition~\ref{definition:cspstar}.

\bigskip

Let $k_\alpha$ be a local field of characteristic $0$ and let $G$ be a semisimple Lie group over $k_\alpha$. The Killing form on the Lie algebra gives rise to a canonical Haar measure  $\mu^{\dagger}_G$ on $G$ (see Section~\ref{sec:measures}); we refer to this measure as the Killing measure of $G$.

\begin{definition}\label{def:renormalized-measure}
Let $G$ be an algebraically simply-connected semisimple Lie group.  We define the \emph{renormalized Killing measure} on $G$ as
\[
     \mu^\diamond_G = \delta \prod_{\alpha \in A} c_\alpha^{-1} \mu^\dagger_{\mathbf{G}_\alpha(k_\alpha)}
\]
where $c_\alpha$ is the Killing measure of the compact real form of $\mathbf{G}_\alpha$ if $\alpha$ is archimedean and otherwise $c_\alpha = 1$. Here $\delta = 2$ if all $k_\alpha$ are archimedean and there is some $\alpha \in A$ such that $k_\alpha=\bbR$ and $\mathbf{G}_\alpha(k_\alpha)$ is not topologically simply-connected. In all other cases, we define $\delta = 1$.
\end{definition}

\begin{theorem} \label{thm:main-lattices}
Let $G$ and $H$ be algebraically simply-connected semisimple Lie groups of rank at least two without compact factors.
Suppose $\Gamma \subseteq G$ and $\Delta \subseteq H$ are irreducible lattices with CSP* such that $\widehat{\Gamma} \cong \widehat{\Delta}$.  Then the renormalized Killing covolumes of $\Gamma$ and $\Delta$ are equal.
\end{theorem}

Again, since CSP* is known if the algebraic group is isotropic, the Harish-Chandra Theorem gives the following unconditional result.

\begin{theorem} \label{thm:non-uniform}
Let $G$ and $H$ be algebraically simply-connected semisimple Lie groups of rank at least two without compact factors.
Suppose $\Gamma \le G$ and $\Delta \le H$ are irreducible non-uniform lattices with $\widehat{\Gamma} \cong \widehat{\Delta}$.  Then the renormalized Killing covolumes of $\Gamma$ and $\Delta$ are equal.
\end{theorem}

The outline of the article is as follows. In Section~\ref{sec:csp} we recall aspects of the congruence subgroup problem needed to formulate our results and we provide the definition of CSP*.  In Section~\ref{sec:measures}, we discuss the renormalization of Killing measures.  The main results are proven in Section~\ref{sec:main} which ends with a remark on profinite almost rigidity and how our results relate to the work of M.\,Aka~\cite{Aka:arithmetic}.  The article concludes with Appendix~\ref{appendix:s-adelic-superrigidity} on the \(S\)-arithmetic extension of the adelic superrigidity theorem.

\medskip
The authors are grateful to Deutsche Forschungsgemeinschaft (DFG) for financial support in the Priority Program ``Geometry at Infinity'' (DFG 441848266, 441425994).  H.\,K. and S.\,K. acknowledge additional funding from the DFG Research Training Group ``Algebro-Geometric Methods in Algebra, Arithmetic, and Topology'', DFG 284078965. R.\,K. acknowledges additional funding via the DFG Individual Research Grant KO 4323/15-1. The authors thank Andrei Rapinchuk for valuable discussions concering CSP.

\section{Around the congruence subgroup property}
\label{sec:csp}

To fix the notation and to define CSP* concisely, we review some aspects of the congruence subgroup property.  Let $k$ be a number field and let $S \subset V(k)$ be a finite subset that contains all infinite places.   We still denote the ring of $S$-integers in $k$ by $\mathcal{O}_{k,S}$.  Let $\mathbf{G}$ be a simply-connected absolutely almost simple linear algebraic $k$-group.  Picking an embedding $\mathbf{G} \subset \mathbf{GL_n}$, the $\mathcal{O}_{k,S}$-points $\mathbf{G}(\mathcal{O}_{k,S})$ are defined.  A subgroup $\Gamma \le \mathbf{G}(\mathcal{O}_{k,S})$ is called a \emph{congruence subgroup} if for some nonzero ideal $\mathfrak{a} \trianglelefteq \mathcal{O}_{k,S}$, the group $\Gamma$ contains the kernel of the reduction homomorphism $\mathbf{G}(\mathcal{O}_{k,S}) \rightarrow \mathbf{G}(\mathcal{O}_{k,S}/\mathfrak{a})$.  Taking the congruence subgroups as a unit neighborhood base defines the \emph{congruence topology} on $\mathbf{G}(k)$.  It is a priori coarser than the \emph{arithmetic topology} on $\mathbf{G}(k)$ consisting of all finite index subgroups of $\mathbf{G}(\mathcal{O}_{k,S})$.  Both topologies are independent of the chosen embedding $\mathbf{G} \subset \mathbf{GL_n}$.  With respect to the canonical uniform structure on $\mathbf{G}(k)$, the arithmetic topology defines the completion $\widehat{\mathbf{G}(k)}$ and the congruence topology defines the completion $\overline{\mathbf{G}(k)}$.  We have a canonical surjective homomorphism $\widehat{\mathbf{G}(k)} \rightarrow \overline{\mathbf{G}(k)}$.  The kernel of this homomorphism is denoted by $C(\mathbf{G},S)$ and is called the \emph{congruence kernel} of $\mathbf{G}$ with respect to $S$.  We say that $\mathbf{G}$ \emph{has CSP with respect to $S$} if the profinite group $C(\mathbf{G}, S)$ is actually finite.

The congruence kernel is closely related to the \emph{metaplectic kernel} $M(\mathbf{G}, S)$ of $\mathbf{G}$, where for any subset $V \subset V(k)$ we define
\[ M(\mathbf{G}, V) = \ker [H_m^2(\mathbf{G}(\mathbb{A}_{k, V}), I) \longrightarrow H^2(\mathbf{G}(k), I)]. \]
Here $I = \R/\Z$ denotes the one-dimensional real torus, $H_m^2$ refers to cohomology defined by measurable cocycles, and $\mathbb{A}_{k,V}$ is the ring of \emph{$V$-adeles} consisting of all elements in the product $\prod_{v \notin V} k_v$ almost all of whose coordinates lie in the ring of integers $\mathcal{O}_v \subset k_v$.  We will later write $\mathbb{A}_k = \mathbb{A}_{k, \emptyset}$ and $\mathbb{A}_k^f = \mathbb{A}_{k, V_\infty}$ where $V_\infty$ is the set of infinite places of $k$.  In fact, as explained in~\cite[Theorem 2]{Prasad-Rapinchuk:survey}, it is known that if $\mathbf{G}$ satisfies $\sum_{v \in S} \rank_{k_v} \mathbf{G} \ge 2$ and $\rank_{k_v} \mathbf{G} \ge 1$ for all finite $v \in S$ and if $\mathbf{G}$ has CSP with respect to $S$, then
\[
	C(\mathbf{G},S) \text{ is isomorphic to the Pontrjagin dual of } M(\mathbf{G}, S)
\]
provided the following is true.

\begin{conjecture}[Margulis--Platonov]
Let $T$ be the (possibly empty) set of finite places of $k$ at which $\mathbf{G}$ is $k_v$-anisotropic.  If $G_T = \prod_{v \in T} \mathbf{G}(k_v)$ denotes the corresponding profinite group, then every non-central normal subgroup $N \trianglelefteq \mathbf{G}(k)$ is given by intersecting the diagonally embeded $\mathbf{G}(k) \le G_T$ with an open normal subgroup $W \trianglelefteq G_T$.
\end{conjecture}

So the conjecture states in particular that $\mathbf{G}(k)$ does not admit non-central normal subgroups if $\mathbf{G}$ is isotropic at all finite places.  Note that this is automatic if $\mathbf{G}$ does not have type $A_n$~\cite{Kneser:galois}*{Satz~3}.

The metaplectic kernel has been calculated in general by Prasad and Rapinchuk~\cite{Prasad-Rapinchuk:metaplectic} with the caveat that ``Conjecture (U)'' has to be assumed if $\mathbf{G}$ is \emph{special}, meaning $\mathbf{G}$ is a type ${}^2 A_n$-group given by $\mathbf{G} = \mathbf{SU}(h)$ where $h$ is a non-degenerate hermitian form over a non-commutative division algebra with involution of the second kind.  To state the conjecture, for any finite set $V \subset V(k)$, we define
\[ M_V(\mathbf{G}) = \ker\left[H_m^2(\textstyle\prod_{v \in V} \mathbf{G}(k_v), I) \longrightarrow H^2(\mathbf{G}(k), I)\right]. \]
Note that for $V \subset W$, we get a canonical inclusion $M_V(\mathbf{G}) \subseteq M_W(\mathbf{G})$.

\begin{conjectureU}[see \cite{Prasad-Rapinchuk:metaplectic}] Suppose $\mathbf{G}$ is special and $V \subset V(k)$ is finite.  If $V_0 \subset V$ is the subset of finite places, then $M_V(G) = M_{V_0}(G)$.
\end{conjectureU}

Now we can state the calculation of metaplectic kernels by Prasad--Rapinchuk.  For the reader's convenience, we slightly reformulate the result needed for our purpose and we comment on the proof.  This is well-known, but we could not locate this exact statement in the literature.  The group of roots of unity in $k$ is denoted by $\mu(k) \le k^*$.

\begin{theorem}\label{thm:metaplectic}
Assume $\mathbf{G}$ is isotropic at every finite place in $S$.  In case $\mathbf{G}$ is special, assume moreover Conjecture~(U) holds true.  If $S$ consists of the infinite places only and $\mathbf{G}_v$ is topologically simply-connected at every real place, then the metaplectic kernel $M(\mathbf{G}, S)$ is isomorphic to $\mu(k)$.  In all other cases, $M(\mathbf{G}, S)$ is trivial.
\end{theorem}

\begin{proof}
  The main theorem of~\cite{Prasad-Rapinchuk:metaplectic} by Prasad--Rapinchuk states explicitly that in the remaining cases $M(\mathbf{G}, S)$ is trivial.  If on the other hand, $S$ is the set of infinite places and $\mathbf{G}$ is topologically simply-connected at all real places, then the real Lie group
  \[ G = \prod_{v \in S} \mathbf{G}(k_v) \]
  is simply-connected because at complex places the notions of algebraic and topological simply-connectedness coincide.  Therefore, we have $H^2_m(G,I) \cong H^2_m(G, \R)$ as proven by Moore in~\cite{Moore:extensions}*{Theorem~A}.  The latter can be computed by the Lie algebra cohomology $H^2_m(G, \R) \cong H^2(\mathfrak{g}, \mathfrak{k}; \R)$ according to a result of Wigner \cite{Wigner:algebraic-cohomology}*{Corollary to Theorem~3}, where $\mathfrak{k}$ is the Lie algebra of a maximal compact subgroup $K \le G$.   By the classical work of Chevalley, Eilenberg, and Cartan, we have $H^2(\mathfrak{g}, \mathfrak{k}; \R) \cong H^2(G_u/K;\R)$, where $G_u$ is the compact form of the complexified Lie group $G_\C \cong \prod_{v \mid \infty} \operatorname{Res}_{k_v/\R} \mathbf{G}_v(\C)$.  Since $G_u$ is a deformation retract of $G_\C$, while $K$ is a deformation retract of $G$, both $G_u$ and $K$ are 2-connected, so the exact sequence of homotopy groups of the fibration
  \[ 1 \longrightarrow K \longrightarrow G_u \longrightarrow G_u/K \longrightarrow 1 \]
  shows that $\pi_2(G_u/K) \cong \pi_1 (G_u/K) \cong 1$.  Finally, the Hurewicz theorem in combination with the universal coefficient theorem imply that $H^2(G_u/K, \R) \cong 0 $ so that $H^2_m(G, I) \cong 0$.  Moore's work also shows that $H^1_m(G, I) \cong H^1_c(G, I)$ coincides with the first group cohomology with respect to continuous cocycles~\cite{Moore:group-extensions}*{Theorem~3 and Corollary~1}.  As $G$ acts trivially on $I$ and since $I$ is abelian,  $H^1_c(G,I)$ is given by continuous group homomorphisms $G \rightarrow I$ and any such homomorphism must be trivial because $G$ is semisimple.  So also $H^1_m(G,I) \cong 0$.

  Therefore, the inflation-restriction exact sequence for measurable cohomology associated with the short exact sequence
  \[ 1 \longrightarrow G \longrightarrow \mathbf{G}(\mathbb{A}_k) \longrightarrow \mathbf{G}(\mathbb{A}_k^f) \longrightarrow 1 \]
  shows that the inflation map 
  \[ H_m^2(\mathbf{G}(\mathbb{A}^f_k), I) \longrightarrow H_m^2(\mathbf{G}(\mathbb{A}_k), I) \]
induced by the projection $\mathbf{G}(\mathbb{A}_k) \rightarrow \mathbf{G}(\mathbb{A}^f_k)$ is an isomorphism.  By functoriality, the kernels of the restriction maps to $H^2(\mathbf{G}(k), I)$ are thus also isomorphic, meaning $M(\mathbf{G}, S) \cong M(\mathbf{G}, \emptyset)$.  Under the assumption of Conjecture~$(U)$ if $\mathbf{G}$ is special, the Prasad--Rapinchuk theorem now gives $M(\mathbf{G}, \emptyset) \cong \mu(k)$.
\end{proof}

In order to have these calculations fully available, we give the following definition.

\begin{definition} \label{definition:cspstar} We say that $\mathbf{G}$ \emph{has CSP* with respect to $S$} if
  \begin{enumerate}
  \item $\mathbf{G}$ has CSP with respect to $S$,
  \item \label{item:cequalsm} $|C(\mathbf{G}, S)| = |M(\mathbf{G}, S)|$, 
  \item \label{item:conjecture-u} Conjecture~(U) holds true if $\mathbf{G}$ is special.
  \end{enumerate}
  Let $\Gamma$ be an irreducible lattice in an algebraically simply-connected semisimple Lie group $G$ of rank at least two and without compact factors and let $\mathbf{G}$ be the unique simple algebraic group in which $\Gamma$ is $S$-arithmetic given by Margulis' arithmeticity theorem.  Then we say \emph{$\Gamma$ has CSP*} if $\mathbf{G}$ has CSP* with respect to~$S$.
\end{definition}

Assuming $\sum_{v \in S} \rank_{k_v} \mathbf{G} \ge 2$ and $\rank_{k_v} \mathbf{G} \ge 1$ for all finite places $v \in S$,  we just explained that \eqref{item:cequalsm} holds true if the Margulis--Platonov conjecture is true for $\mathbf{G}$.  In that case, finiteness of $C(\mathbf{G},S)$ is actually equivalent to centrality \cite{Prasad-Rapinchuk:survey}*{Theorem~2} and a well-known conjecture of Serre \cite{Platonov-Rapinchuk:algebraic-groups}*{(9.45), p.\,556} says that this should be true under our assumption on the ranks.

According to \cite{Prasad-Rapinchuk:survey}*{Section~3.5}, the Margulis--Platonov conjecture has been established in all cases except for certain anisotropic outer forms of type $A_n$, anisotropic forms of type ${}^3 D_4$ and ${}^6 D_4$, and certain anisotropic type $E_6$ forms.  Similarly, according to \cite{Prasad-Rapinchuk:survey}*{Section~5.1}, centrality of $C(\mathbf{G},S)$ is open for all anisotropic inner forms of type $A_n$, ``most'' anisotropic outer forms of type $A_n$, anisotropic forms of type ${}^3 D_4$ and ${}^6 D_4$, and ``most'' anisotropic groups of type $E_6$.  To summarize:

\begin{theorem} \label{thm:status-of-csp}
  Suppose $\sum_{v \in S} \rank_{k_v} \mathbf{G} \ge 2$ and $\rank_{k_v} \mathbf{G} \ge 1$ for all finite places $v \in S$.  Then the group $\mathbf{G}$ has CSP* with respect to $S$ if either
  \begin{enumerate}
  \item $\mathbf{G}$ is $k$-isotropic, or
  \item $\mathbf{G}$ is $k$-anisotropic but not of type $A_n$, not a triality form of type $D_4$, and not of exceptional type $E_6$.
  \end{enumerate}
\end{theorem}

Recall that according to the Harish-Chandra Theorem, an irreducible lattice $\Gamma$ in an algebraically simply-connected semisimple Lie group $G$ is uniform if and only if the corresponding unique simple algebraic group $\mathbf{G}$ is $k$-anisotropic.  This explains how Theorem~\ref{thm:non-uniform} follows from Theorem~\ref{thm:main-lattices}.

Finally, we will explain in the proof of Theorem~\ref{thm:same-lie-group} that any irreducible lattice \(\Gamma \le G\) in a higher rank semisimple Lie group with finite center and without compact factors has a finite index subgroup \(\Gamma' \le \Gamma\) which embeds as an irreducible lattice in a uniquely determined algebraically simply-connected semisimple Lie group of rank at least two without compact factors.  It is then well-defined to say that \(\Gamma\) has CSP* if \(\Gamma'\) has CSP*.

\section{Measures on algebraic groups} \label{sec:measures}

In this section we fix measures on algebraically simply-connected semisimple Lie groups.
To this end, let $(K,|\cdot|)$ be a local field of characteristic $0$ with absolute value $|\cdot|$.
We normalize a Haar measure $\mu_K$ on $K$ in the following way:
\begin{itemize}
\item $K = \bbR$: fix $\mu_\bbR([0,1]) = 1$
\item $K = \bbC$:  fix $\mu_\bbC([0,1]+i[0,1]) = 2$
\item $K$ nonarchimedean with ring of integers $\mathcal{O}$: fix $\mu_K(\mathcal{O}) = 1$.
\end{itemize}

Let $G$ be a semisimple Lie group over $K$ of dimension $d$ and let $\fg$ denote its Lie algebra. Every non-trivial form $\omega_0 \in \bigwedge^d \fg^*$ extends to a unique left-invariant top-degree differential form $\omega$ on $G$ and induces a left Haar measure $\mu_\omega$ (that depends on the fixed Haar measures $\mu_K$ on $K$).
As $\fg$ is semisimple and $K$ is of characteristic $0$, the Killing form $B_\fg$ on $\fg$ is non-degenerate. The Killing form can be used to define a canonical normalization of the Haar measure $\mu^{\dagger}_G$. 
Let $e_1,\dots, e_d$ be an ordered basis of $\fg$ and let $e_1^*,\dots,e_n^*$ denote the associated ordered dual basis. For $\omega_0 = e_1^*\wedge e_2^* \wedge \dots \wedge e_n^*$ we define the \emph{Killing measure}
\[
	\mu^{\dagger}_G = \sqrt{|\det\bigl((B(e_i,e_j))_{i,j}\bigr)|} \; \mu_{\omega}.
\]
This is independent of the chosen ordered basis. Indeed, if the base change is given by a transformation matrix $T$, then $\omega_0$ transforms with the inverse determinant of $T$ and $\det\bigl((B(e_i,e_j))_{i,j}\bigr)$ transforms with $\det(T)^2$.

\begin{lemma}\label{lem:tamagawa-to-killing}
Let $\mathbf{G}$ be a connected semisimple $d$-dimensional linear algebraic group over a number field $k$ with discriminant~$d_k$.  Then the Tamagawa measure $\tau$ and the product measure $\mu^\dagger = \prod_{v \in V(k)} \mu^\dagger_{\mathbf{G}(k_v)}$ on $\mathbf{G}(\bbA_k)$ satisfy
\[
	 |d_k|^{-d/2} \mu^\dagger = \tau.
\]
\end{lemma}

\begin{proof}
  Let $\fg$ be the $k$-Lie algebra of $\mathbf{G}$.  We pick an ordered basis $e_1,\dots, e_d$ of $\fg$ so that the multilinear form $\omega_0 = e_1^* \wedge \dots \wedge e_d^*$ extends to a left-invariant algebraic top-dimensional $k$-form $\omega$ on $\mathbf{G}$.  Then for every $v \in V(k)$, the form $\omega$ extends in turn to a differential $k_v$-form on $\mathbf{G}(k_v)$ and gives us a measure $\mu_{\omega,v}$ as explained above.  The Tamagawa measure $\tau$ is then given by $\tau = |d_k|^{-d/2} \prod_{v \in V(k)} \mu_{\omega,v}$.  For a different basis, the differential form $\omega$ is scaled by the determinant of the transformation matrix and we have the product formula $\prod_{v \in V(k)} |a|_v = 1$ for all $a \in k$, so $\tau$ is well-defined.  Since $\det B(e_i, e_j) \in k$, the product formula also gives
  \[ \prod_{v \in V(k)} \mu_{\omega,v} = \prod_{v \in V(k)} \sqrt{|\det B(e_i, e_j)|_v} \ \mu_{\omega, v} = \prod_{v \in V(k)} \mu^\dagger_{\mathbf{G}(k_V)} = \mu^\dagger, \]
  whence the assertion.
\end{proof}
Recall that we defined the renormalized Killing volume in Definition~\ref{def:renormalized-measure} in the introduction.
\begin{lemma} \label{lem:forget-compact-factors}
Let $G$ be an algebraically simply-connected semisimple Lie group and let $\pi\colon G \to G^{\text{is}}$ be the canonical projection to the product of the noncompact factors. Assume the kernel of $\pi$ is a real Lie group.  If $\Gamma \subseteq G$ is a torsion-free lattice, then the renormalized Killing covolume of $\Gamma$ in $G$ equals the renormalized Killing covolume of $\pi(\Gamma)$ in~$G^{\text{is}}$
 \end{lemma}
 
 \begin{proof}
   We first note that the factors $\delta$ in the renormalized Killing measures for $G$ and $G^{\text{is}}$ agree.  Indeed, if $\mathbf{G}_{\alpha}$ is anisotropic and $k_\alpha = \bbR$, then $\mathbf{G}_{\alpha}(k_\alpha)$ is topologically simply-connected, being a deformation retract of $\mathbf{G}_\alpha(\C)$ which is topologically simply-connected because it is algebraically so.
   
 By assumption, we have $G = K \times G^{\text{is}}$ for a compact Lie group~$K$.  Since $\Gamma$ is torsion-free and discrete, we have $\Gamma \cap K = \{1\}$.  We choose a measurable fundamental domain $\mathcal{F}$ for the action of $\Gamma$ on $G^{\text{is}}$.  Then $K \times \mathcal{F}$ is a measurable fundamental domain for $\Gamma$ in $G$.
 Hence the ratio of the covolumes of $\Gamma$ and $\pi(\Gamma)$ is the renormalized Killing volume of $K$, which is one by construction.
 \end{proof}

 \section{Proof of the main theorem} \label{sec:main}

 Let $k$ be a number field and let $S \subset V(k)$ be a finite subset containing all infinite places.  We extend the terminology in \cite{KammeyerKionke}*{Definition~3.1} and say that a simply-connected absolutely almost simple $k$-group $\mathbf{G}$ is \emph{$S$-algebraically superrigid} if for every field $l$ of characteristic zero, every $S$-arithmetic subgroup $\Gamma$ of $\mathbf{G}$, and every homomorphism $\delta \colon \Gamma \rightarrow \mathbf{H}(l)$ to the $l$-points of a connected and absolutely almost simple $l$-group $\mathbf{H}$ with Zariski dense image, there exists a unique homomorphism $\sigma \colon k \rightarrow l$, a unique $l$-epimorphism $\eta \colon {}^\sigma \mathbf{G} \rightarrow \mathbf{H}$, and a unique homomorphism $\nu \colon \Gamma \rightarrow Z(\mathbf{H})(l)$ such that $\delta(\gamma) = \nu(\gamma)\eta(\gamma)$ for all $\gamma \in \Gamma$. By a standard descent argument (e.g.~using \cite[Lemma 2.2]{KammeyerKionke}), it suffices to verify this condition for $\ell = \bbC$.

 The \emph{Margulis superrigidity theorem} \cite{Margulis:discrete-subgroups}*{Theorem~(C), p.\,259} asserts that $\mathbf{G}$ is $S$-algebraically superrigid if $\sum_{v \in S} \rank_{k_v} \mathbf{G} \ge 2$.  We outline in Appendix~\ref{appendix:s-adelic-superrigidity} that the \emph{adelic superrigidity} theorem that the first two authors concluded from this property in~\cite{KammeyerKionke}*{Theorem~3.2} extends effortlessly to $S$-algebraically superrigid groups: if $l$ is a number field, then also homomorphisms $\Gamma \rightarrow \mathbf{H}(\mathbb{A}_{l,T})$ from an $S$-arithmetic subgroup $\Gamma$ of $\mathbf{G}$ to the $T$-adele points of $\mathbf{H}$ extend to a homomorphism of group schemes over the adele points.  This theorem will be key in our argument.  Two number fields are called \emph{arithmetically equivalent} over $\Q$ if almost all rational primes have the same (unordered) tuple of inertia degrees in both number fields.  The \emph{congruence hull} $\Gamma^c$ of an $S$-arithmetic group $\Gamma$ is the smallest congruence subgroup containing~$\Gamma$.
 
\begin{proposition}\label{prop:S-arithmetic-volume-base}
Let $\mathbf{G}$ and $\mathbf{H}$ be simply-connected absolutely almost simple groups defined over number fields $k$ and $l$.   Let $\Gamma \subseteq \mathbf{G}(k)$ be an $S$-arithmetic and let $\Delta \subseteq \mathbf{H}(l)$ be a $T$-arithmetic group such that there exists an isomorphism $\Psi \colon \widehat{\Gamma} \rightarrow \widehat{\Delta}$.  Assume that $\mathbf{G}$ and $\mathbf{H}$ are $S$- and $T$-algebraically superrigid, respectively.  Then the following hold:
\begin{enumerate}
 \item \label{it:arithmetically-equivalent} The fields $k$ and $l$ are arithmetically equivalent over $\bbQ$.
 \item\label{it:finite-places} $S$ contains a finite place if and only if $T$ contains a finite place.
 \item There is an open normal subgroup $U \subseteq \widehat{\Gamma}$ such that the congruence hulls $\Gamma_0^c$ of $\Gamma_0 := U \cap \Gamma$ and $\Delta_0^c$ of $\Delta_0 = \Psi(U) \cap \Delta$
have equal Killing covolumes in $\prod_{v \in S} \mathbf{G}(k_v)$ and resp. $\prod_{w\in T} \mathbf{H}(l_w)$.
\end{enumerate}
\end{proposition}

\begin{proof}
Let $\Psi \colon \widehat{\Gamma} \to \widehat{\Delta}$ be an isomorphism.  Using Theorem~\ref{thm:adelic-rigidity-s-arithmetic} and the existence of $\Psi$ and its inverse as in the proof of Theorem~3.4 in \cite{KammeyerKionke}, we deduce that there is an isomorphism $j_1 \colon \mathbb{A}_{l, T} \to \mathbb{A}_{k,S}$ of adele rings and an isomorphism of group schemes $\eta_1 \colon \alg{G} \times_k \mathbb{A}_{k,S} \to \alg{H} \times_l \mathbb{A}_{l,T}$ over $j_1$ and an open normal subgroup $U \normal_o \widehat{\Gamma}$ such that
 \begin{equation*}
      \xymatrix{
      	U \ar[d]_{q_\alg{G}}\ar[r]^{\cong}_{\Psi} & \Psi(U) \ar[d]_{q_{\alg{H}}}\\
        \alg{G}(\mathbb{A}_{k, S}) \ar[r]^{\cong}_{\eta_1} & \alg{H}(\mathbb{A}_{l, T}),
        }
\end{equation*}
commutes. Here $U$ is any open normal subgroup contained in the kernels of the homomorphisms $\hat{\nu}_1 \colon \widehat{\Gamma} \to Z(\mathbf{H})(\mathbb{A}_{l, T})$ and $\hat{\nu}_2\circ \Psi \colon \widehat{\Gamma} \to Z(\alg{G})(\mathbb{A}_{k, S})$ induced on profinite completions by the homomorphisms $\nu_1$ and $\nu_2$ coming from Theorem~\ref{thm:adelic-rigidity-s-arithmetic}.  We write $\widetilde{U} =q_{\alg{G}}(U)$.

The existence of the isomorphism $\mathbb{A}_{l, T} \to \mathbb{A}_{k,S}$ implies that almost all rational primes have the same decomposition type in $k$ and $l$, which gives \eqref{it:arithmetically-equivalent}, and that $S$ contains a place over $p$ if and only if $T$ does, which shows~\eqref{it:finite-places}.

Let $\Gamma_0^c$ and $\Delta_0^c$ denote the congruence hulls of $\Gamma_0$ and $\Delta_0$, respectively, so $\Gamma_0^c = \mathbf{G}(k) \cap \widetilde{U}$ and $\Delta_0^c = \mathbf{H}(l) \cap \eta_1(\widetilde{U})$.  Kottwitz proved \cite{Kottwitz} that the Tamagawa number $\tau$ equals $1$ for simply-connected semisimple groups (the missing Hasse principle for $E_8$ was later established by Chernousov~\cite{Chernousov}).  Hence, Lemma \ref{lem:tamagawa-to-killing} implies that the Killing covolume of $\mathbf{G}(k)$ in $\mathbf{G}(\bbA_k)$ is $|d_k|^{\frac{\dim(\mathbf{G})}{2}}$.

Let $\mathcal{F}$ be a measurable fundamental domain for the action of $\Gamma^c_0$ on $\prod_{v\in S} \mathbf{G}(k_v)$.  Then it is straightforward to check using strong approximation that $\mathcal{F} \times \widetilde{U}$ is a measurable fundamental domain for the action of $\mathbf{G}(k)$ on $\mathbf{G}(\bbA_k)$ (cf.~\cite[\S 4.2]{Ono:algebraicgroups}). In turn, we obtain
\[
	\vol_{\mu^\dagger_S}\left(\prod_{v \in S} \mathbf{G}(k_v)/\Gamma_0^c\right) = |d_k|^{\frac{\dim(\mathbf{G})}{2}} \vol_{\mu^\dagger_{V(k) \setminus S}}(\widetilde{U})^{-1}
\]
where $\mu^\dagger_S = \prod_{v \in S} \mu^\dagger_{\mathbf{G}(k_v)}$ and $\mu^\dagger_{V(k) \setminus S} = \prod_{v \notin S} \mu^\dagger_{\mathbf{G}(k_v)}$.  The same argument gives a formula for the Killing covolume of $\Delta_0^c$ in $\prod_{w \in T} \mathbf{H}(l_w)$ as a product of $|d_l|^{\frac{\dim(\mathbf{H})}{2}}$ and $\vol_{\mu^\dagger_{V(l) \setminus T}}(q_{\alg{H}}(\Psi(U)))^{-1}$.  It follows (for instance) from adelic superrigidity, that $\dim(\mathbf{G}) = \dim(\mathbf{H})$.  The number fields $l$ and $k$ are arithmetically equivalent over $\bbQ$, and thus have the same discriminant; see \cite[Theorem (1.4)]{Klingen:similarities}.  Since an isomorphism of Lie algebras maps the Killing forms to one another, the isomorphim $\eta_1$ preserves the Killing covolume so that
\[
 \vol_{\mu^\dagger_{V(k) \setminus S}}(\widetilde{U}) = \vol_{\mu^\dagger_{V(l) \setminus T}}(\eta_1(\widetilde{U})) = \vol_{\mu^\dagger_{V(l) \setminus T}}(q_{\alg{H}}(\Psi(U))).
 \]
 We deduce that the Killing covolumes of $\Gamma_0^c$ in $\prod_{v\in S} \mathbf{G}(k_v)$ and of $\Delta_0^c$ in $\prod_{w\in T} \mathbf{H}(l_w)$ coincide.
\end{proof}

\begin{corollary}\label{cor:S-arithmetic-volume-H1}
In the setting of Proposition \ref{prop:S-arithmetic-volume-base} assume that 
 \[\Hom(H_1(\Gamma,\bbZ), Z(\mathbf{H})(\bbC)) = \{1\}.\]
Then the Killing covolumes of the congruence hulls $\Gamma^c$ in $\prod_{v \in S} \mathbf{G}(K_v)$ and $\Delta^c$ in $\prod_{w\in T} \mathbf{H}(L_w)$ are equal.
\end{corollary}
\begin{proof}
As the abelianization is a profinite invariant, we have $H_1(\Gamma,\bbZ) \cong H_1(\Delta,\bbZ)$.  Moreover, $\alg{G}$ and $\alg{H}$ are isomorphic over $\bbC$, so that  $Z(\alg{G})(\bbC) \cong Z(\alg{H})(\bbC)$.
This shows that if $\Hom(H_1(\Gamma,\bbZ), Z(\mathbf{H})(\bbC))$ is trivial, then one can choose $U = \widehat{\Gamma}$ in the proof of Proposition \ref{prop:S-arithmetic-volume-base} (compare to the proof of \cite[Theorem 3.4]{KammeyerKionke}). Hence $\Gamma = \Gamma_0$ and $\Delta = \Delta_0$ and the assertion follows.
\end{proof}

In particular, this result applies to groups of type $F_4$ because they have trivial center.

\begin{corollary}\label{cor:F4}
In the setting of~\ref{prop:S-arithmetic-volume-base}, assume $\mathbf{G}$ and $\mathbf{H}$ have type~$F_4$.  Then the Killing covolumes of the congruence hulls $\Gamma^c$ in $\prod_{v \in S} \mathbf{G}(K_v)$ and $\Delta^c$ in $\prod_{w\in T} \mathbf{H}(L_w)$ are equal.
\end{corollary}

Theorem~\ref{thm:octonionic} is the special case when $k$ and $l$ are totally real, $S$ and $T$ consist of the infinite places only, $\mathbf{G}$ and $\mathbf{H}$ are type $F_4$ forms which are isomorphic to $F_{4(-20)}$ at one real place and anisotropic at all other real places, and $\Gamma$ and $\Delta$ are congruence subgroups.  Note that $\mathbf{G}$ and $\mathbf{H}$ are $S$- and $T$-algebraically superrigid even though they have rank one.
 As mentioned in the first paragraph of this section, superrigidity over $\mathbb{C}$ entails algebraic superrigidity. This can be deduced from the work of Corlette~\cite{Corlette:archimedean-superrigidity} and Gromov-Schoen~\cite{GromovSchoen} following along the lines of Margulis \cite{Margulis:discrete-subgroups}.
 
\medskip

In the proof of the next result, it will become apparent why renormalization of Killing measures is necessary.

\begin{theorem} \label{thm:S-arithmetic-main}
Let $\mathbf{G}$ and $\mathbf{H}$ be simply-connected absolutely almost simple algebraic groups defined over number fields $k$ and $l$.

Let $\Gamma \subseteq \mathbf{G}(k)$ be $S$-arithmetic and let $\Delta \subseteq \mathbf{H}(l)$ be $T$-arithmetic with $\widehat{\Gamma} \cong \widehat{\Delta}$.  Assume $\mathbf{G}$ and $\mathbf{H}$ are algebraically superrigid and have CSP* with respect to $S$ and $T$, respectively.  Then the renormalized Killing covolumes of $\Gamma$ in $\prod_{v \in S} \mathbf{G}(k_v)$ and $\Delta$ in $\prod_{w\in T} \mathbf{H}(l_w)$ are equal.
\end{theorem}
\begin{proof}
As before, we denote by $\Psi \colon \widehat{\Gamma} \to \widehat{\Delta}$ a fixed isomorphism.  We may apply Proposition \ref{prop:S-arithmetic-volume-base} to find a suitable open normal subgroup $U \trianglelefteq \widehat{\Gamma}$.  We may shrink $U$ to achieve that $U \cap C(\alg{G},S) = \{1\}$ and $\Psi(U)\cap C(\alg{H},T) = \{1\}$.
As usual, we define $\Gamma_0 = \Gamma \cap U$ and $\Delta_0 = \Delta \cap \Psi(U)$.
As
\[
	|\Gamma:\Gamma_0| = |\widehat{\Gamma}:U| = |\widehat{\Delta}:\Psi(U)| = |\Delta:\Delta_0|
\]
 it is sufficient to show that
the subgroups $\Gamma_0$ and $\Delta_0$ have the same renormalized Killing covolumes.
Since $U \cap C(\alg{G},S) = \{1\}$, we have
\[|\Gamma_0^c:\Gamma_0| = |C(\alg{G},S)| = |M(\alg{G}, S)|\]
and similarly $|\Delta_0^c:\Delta_0| = |M(\alg{H}, T)|$.
Proposition~\ref{prop:S-arithmetic-volume-base} gives that the Killing covolumes of $\Delta_0^c$ and $\Gamma_0^c$ agree, that $S$ contains a finite place if and only if $T$ does, and that $k$ and $l$ are arithmetically equivalent.  In particular $[k:\bbQ] = [l:\bbQ]$, and $k$ and $l$ have the same number of real and complex places; see \cite[Theorem (1.4)]{Klingen:similarities}.  Therefore, the rescaling of the Killing volume by the Killing volumes of the compact real forms is the same for $k$ and $l$ and can be disregarded.  Moreover, if $S$ and $T$ contain a finite place, then Theorem~\ref{thm:metaplectic} implies $|M(\alg{G}, S)|= |M(\alg{H},T)| = 1$ and the assertion follows.

So suppose now that $S$ and $T$ consist of the archimedean places of $k$ and~$l$ only. 
As we just said, $k$ is totally imaginary if and only if $l$ is.  In this case $|M(\alg{G}, S)|= |\mu(k)| = |\mu(l)| = |M(\alg{H}, T)|$ because arithmetically equivalent fields contain the same roots of unity; see \cite[Theorem (1.4)]{Klingen:similarities}.

So finally assume that $k$ and $l$ have at least one real place.  By Theorem~\ref{thm:metaplectic}, we have $|M(\alg{G}, S)| = 1$ if and only if $\alg{G}(k_v)$ is not topologically simply-connected at some real place $v$ and then the factor $\delta_G$ in the renormalized Killing measure is given by $\delta_G = 2$.  If on the other hand $\alg{G}(k_v)$ is topologically simply-connected at all real places $v$, then $|M(\alg{G}, S)| = 2$ and $\delta_G = 1$.  So in both cases $\delta_G \cdot |M(\alg{G}, S)| = 2$ and the same applies to $\mathbf{H}$.  It follows that the renormalized Killing covolumes of $\Gamma_0$ and $\Delta_0$ agree.
\end{proof}

\begin{example}
  As an addendum to the last proof, we give examples of \(\mathbf{G}\) and \(\mathbf{H}\) whose metaplectic kernels are different to show that the factor \(\delta\) in the renormalization of Killing measures cannot be ommitted.  In fact, we can use M.\,Aka's examples of spinor groups which exhibit that Kazhdan's property (T) is not a profinite property~\cite{Aka:property-t}*{Theorem~1}.

  Consider the quadratic forms
  \[ f = \langle 1, 1, 1, 1, 1, 1, -1 \rangle \quad \text{and} \quad g = \langle 1, 1, -1, -1, -1, -1, -1 \rangle \]
  and let \(\mathbf{G} = \mathbf{Spin}(f)\) and \(\mathbf{H} = \mathbf{Spin}(g)\), both considered as groups over the number field \(\Q(\sqrt{2})\).  Then \(\mathbf{G}\) is topologically simply-connected at both real places because the maximal compact subgroup \(\operatorname{Spin}(6)\) is a simply-connected deformation retract.  In contrast, \(\mathbf{H}\) is not topologically simply-connected at the real places because at these places, \(\mathbf{H}\) is a twofold covering of \(\operatorname{SO}^0(2,5)\) whose fundamental group is infinite.  For \(S\) consisting of the real places of \(\Q(\sqrt{2})\), Theorem~\ref{thm:metaplectic} gives that \(M(\mathbf{G}, S)\) has order two whereas \(M(\mathbf{H}, S)\) is trivial.  Both \(\mathbf{G}\) and \(\mathbf{H}\) are algebraically superrigid and have CSP* by Theorem~\ref{thm:status-of-csp}.

  Since the quadratic forms \(\langle 1, 1, 1, 1 \rangle\) and \(\langle -1, -1, -1, -1 \rangle\) are isometric over \(\Z_p\) for all primes \(p\), we see that the congruence completions of the arithmetic groups \(\mathbf{G}(\Z[\sqrt{2}])\) and \(\mathbf{H}(\Z[\sqrt{2}])\) (defined as in~\cite{Kammeyer-Sauer:spinor}*{Section~3}) are isomorphic.  It follows that suitable finite index subgroups \(\Gamma\) and \(\Delta\) of these groups have isomorphic profinite completions.
\end{example}
  
\begin{proof}[Proof of Theorem \ref{thm:main-lattices}]
  As we explained below Definition~\ref{def:simply-connected}, we have
  \[ G \cong \prod_{v \in S^{\text{is}}} \mathbf{G}(k_v) \quad \text{and} \quad H \cong \prod_{v \in T^{\text{is}}} \mathbf{H}(l_v) \]
  for simply-connected absolutely almost simple linear algebraic groups $\mathbf{G}$ and $\mathbf{H}$ over $k$ and $l$, respectively, and for finite sets of places $S^{\text{is}}$ and $T^{\text{is}}$ of $k$ and $l$ containing all the infinite places at which $\mathbf{G}$ and $\mathbf{H}$ are isotropic, respectively.  Let $S$ and $T$ be the union of $S^{\text{is}}$ and $T^{\text{is}}$ with all infinite places, respectively.  The group $\Gamma$ is commensurable with $\alg{G}(\mathcal{O}_{k,S})$ and $\Delta$ is commensurable with $\alg{H}(\mathcal{O}_{l,T})$.

Pick $\Gamma_0 \le \Gamma$ of finite index such that $\Gamma_0 \subset \mathbf{G}(\mathcal{O}_{k,S})$ and such that $\Gamma_0$ is torsion-free (Selberg's lemma) and pick $\Delta_0 \le \Delta$ similarly.  Then the open subgroup $U = \overline{\Gamma_0} \cap \overline{\Delta_0}$ of $\widehat{\Gamma} \cong \widehat{\Delta}$ gives torsion-free subgroups $\Gamma_1 = U \cap \Gamma \le \alg{G}(k)$ and $\Delta_1 = U \cap \Delta \le \alg{H}(l)$ such that $[\Gamma : \Gamma_1 ] = [\Delta : \Delta_1]$ and $\widehat{\Gamma_1} \cong \widehat{\Delta_1}$.  The assumption of CSP* gives $|C(\mathbf{G},S)| = |M(\mathbf{G}, S)|$ and similarly for~$\mathbf{H}$.  

Since neither $G$ nor $H$ has compact factors, it follows that neither $S^{\text{is}}$ nor $T^{\text{is}}$ contain infinite places at which $\mathbf{G}$ or $\mathbf{H}$ would be anisotropic, respectively.  As $\Gamma_1$ and $\Delta_1$ are torsion-free, Lemma~\ref{lem:forget-compact-factors} shows that the renormalized Killing covolume of $\Gamma_1 \le G$ and $\Delta_1 \le H$ is the same as the renormalized Killing covolume of $\Gamma_1 \le  \prod_{v \in S} \alg{G}(k_v)$ and $\Delta_1 \le \prod_{v \in T} \alg{H}(l_v)$, respectively.  As we explained in the beginning of the section, the higher rank assumption on $G$ and $H$ implies that $\mathbf{G}$ and $\mathbf{H}$ are $S$- and $T$-algebraically superrigid.  Now the latter renormalized Killing covolumes are equal by Theorem~\ref{thm:S-arithmetic-main} and this completes the proof.
\end{proof}

\begin{proof}[Proof of Theorem~\ref{thm:same-lie-group}.]
  Since the center \(Z(G)\) of \(G\) is finite, there exist finite index subgroups \(\Gamma_0 \le \Gamma\) and \(\Delta_0 \le \Delta\) intersecting \(Z(G)\) trivially.  Fixing the isomorphism \(\widehat{\Gamma} \cong \widehat{\Delta}\), the intersection of the closures \(\overline{\Gamma_0} \cap \overline{\Delta_0}\) in \(\widehat{\Gamma}\) intersects \(\Gamma\) and \(\Delta\) in finite index subgroups \(\Gamma_1\) and \(\Delta_1\) and we have \([\Gamma \colon \Gamma_1] = [\Delta \colon \Delta_1]\).  Since \(\Gamma_1\) and \(\Delta_1\) intersect \(Z(G)\) trivially, the projection map gives embeddings \(\Gamma_1, \Delta_1 \le \operatorname{Ad} G\) into the adjoint form \(\operatorname{Ad} G = G / Z(G)\).

  The adjoint form can be identified with the group of inner automorphisms \(\operatorname{Int} (\mathfrak{g})\) of the Lie algebra \(\mathfrak{g}\) of \(G\).  As such, it is given by the unit path component of the \(\R\)-group \(\operatorname{Aut}(\mathfrak{g})\).  It also is the unit path component of the Zariski connected \(\R\)-group given by the product of the Zariski unit components of the automorphism groups \(\operatorname{Aut} (\mathfrak{g}_i)\) of the simple ideals of \(\mathfrak{g} \cong \mathfrak{g}_1 \oplus \cdots \oplus \mathfrak{g}_r\).  The non-absolutely simple factors in this product decomposition are given by restriction of scalars from an absolutely simple group over \(\C\) and can thus be replaced by these \(\C\)-groups.  Finally, taking the product of all the algebraically simply-connected coverings of the absolutely simple factors, we obtain an algebraically simply-connected Lie group \(G_0\).  Note that \(G_0\) is path connected because the real points of an algebraically simply connected \(\R\)-group are path connected (the \(\C\)-points of the \(\C\)-factors are even topologically simply-connected).  The preimages \(\Gamma_2\) and \(\Delta_2\) of \(\Gamma_1\) and \(\Delta_1\) along the covering morphism \(p \colon G_0 \rightarrow \operatorname{Ad} G\) are irreducible lattices in \(G_0\).  Moreover, \(G_0\) has no anisotropic factor because \(G\) has no compact factor by assumption and similarly \(G_0\) has higher rank because \(G\) does.

  However, we cannot yet ensure that \(\widehat{\Gamma_2} \cong \widehat{\Delta_2}\).  Therefore we argue similarly as above and let \(\Gamma_3 \le \Gamma_2\) and \(\Delta_3 \le \Delta_2\) be finite index subgroups intersecting \(\ker p\) trivially.  Then \(p\) embeds \(\Gamma_3\) and \(\Delta_3\) as finite index subgroups of \(\Gamma_1\) and \(\Delta_1\).  The intersection \(\overline{\Gamma_3} \cap \overline{\Delta_3}\) in \(\widehat{\Gamma_1} \cong \widehat{\Delta_1}\) intersects \(\Gamma_3\) and \(\Delta_3\) in finite index subgroups \(\Gamma_4\) and \(\Delta_4\) and these are finally the irreducible lattices in \(G_0\) with \(\widehat{\Gamma_4} \cong \widehat{\Delta_4}\) and \([\Gamma : \Gamma_4] = [\Delta : \Delta_4]\).

  If we endow \(G_0\) with the Haar measure induced from \(G\) via \(\operatorname{Ad} G\) and denote it likewise by \(\mu\), then by construction
  \[ \frac{\mu(G/\Gamma)}{\mu(G_0 / \Gamma_4)} = \frac{|Z(G)|}{[\Gamma : \Gamma_4] \cdot |\ker p|} = \frac{|Z(G)|}{[\Delta : \Delta_4] \cdot |\ker p|} = \frac{\mu(G/\Delta)}{\mu(G_0 / \Delta_4)}. \]
  Hence the assertion now follows from Theorem~\ref{thm:main-lattices}.
\end{proof}

Since CSP* is known in the isotropic case (Theorem~\ref{thm:status-of-csp}), Theorem~\ref{thm:non-uniform-same-lie-group} is immediate from Theorem~\ref{thm:same-lie-group}.

\begin{remark}[\emph{On the profinite almost rigidity of \(S\)-arithmetic groups}]
  A theorem of Borel~\cite{Borel:bunded-covolume}*{Theorem~4.2} applies to our situation: For every algebraically simply-connected Lie group \(G\) of higher rank without compact factors and with fixed Haar measure \(\mu\) and for every given bound \(c > 0\), there exist only finitely many conjugacy classes of lattices \(\Gamma \le G\) with \(\mu(G/\Gamma) \le c\).

  Moreover any two profinitely isomorphic lattices in such Lie groups define surrounding algebraic groups by arithmeticity which are isomorphic locally almost everywhere by Theorem~\ref{thm:adelic-rigidity-s-arithmetic}.  There are only finitely many possible isomorphism types over each remaining archimedean or \(\mathfrak{p}\)-adic completion.  Thus Borel's theorem and our Theorem~\ref{thm:main-lattices} have the corollary that only finitely many isomorphism types of irreducible lattices in groups \(G\) as above have the same profinite completion if CSP* is granted.  This recovers a theorem of M.\,Aka~\cite{Aka:arithmetic}*{Theorem~2}.

  In fact, in his proof, Aka also uses Borel's theorem above, but only as an intermediate step after showing (in \cite{Aka:arithmetic}*{Section~5}) that \emph{commensurable} \(S\)-arithmetic groups have the same covolume if they are profinitely isomorphic.  This naturally led us to the question whether one can show the profiniteness of volume of higher rank \(S\)-arithmetic groups in general.

  Let us also point out that R.\,Spitler showed in his thesis~\cite{Spitler:thesis} that if a general finitely generated residually finite group \(\Lambda\) is profinitely isomorphic to a higher rank \(S\)-arithmetic group \(\Gamma\), then \(\Lambda\) embeds in a \(T\)-arithmetic group \(\Delta\). If \(\Delta\) has moreover CSP, then \(\widehat{\Gamma} \cong \widehat{\Delta}\).  It is moreover a decades old question, advertised by Platonov and Rapinchuk in~\cite{Platonov-Rapinchuk:algebraic-groups}*{Problem on p.\,424}, whether \(S\)-arithmetic groups with CSP have proper \emph{Grothendieck subgroups}.  In our setting, this just asks whether the inclusion \(\Lambda \le \Delta\) can be proper.  To sum up, assuming CSP* and that the Platonov--Rapinchuk problem has a negative solution, we can conclude that a higher rank \(S\)-arithmetic group \(\Gamma\) is \emph{absolutely profinitely almost rigid}: up to isomorphism, only finitely many finitely generated residually finite groups \(\Lambda\) satisfy \(\widehat{\Gamma} \cong \widehat{\Lambda}\).

  Finally coming back to Question~\ref{question:volume-3-manifolds}, we should mention that Yi Liu~\cite{Liu}*{Theorem~1.1} has recently shown that finite volume hyperbolic 3-manifold groups are profinitely almost rigid among finitely generated 3-manifold groups. However, he writes explicitly~\cite{Liu}*{p.\,743} that he does not know whether all the finitely many hyperbolic 3-manifold groups with the same profinite completion have equal volume.
\end{remark}

\appendix

\section{$S$-adelic superrigidity}
\label{appendix:s-adelic-superrigidity}

We briefly present the extension of adelic superrigidity as proven in \cite{KammeyerKionke} to the $S$-arithmetic case.

\begin{theorem}\label{thm:adelic-rigidity-s-arithmetic}
Let $k$ and $l$ be two algebraic number fields and let $S \subseteq V(k)$ and $T \subseteq V(l)$ be two finite subsets containing all the archimedean places. Assume that $S$ does not contain a finite place at which $\mathbf{G}$ is anisotropic.
Let $\mathbf{G}$ and $\mathbf{H}$ be simply-connected absolutely almost simple groups over $k$ and $l$ and assume that $\mathbf{G}$ is $S$-algebraically superrigid.

Let $\Gamma \subseteq \mathbf{G}(k)$ be an $S$-arithmetic subgroup.  Suppose there is a homomorphism $\phi \colon \Gamma \to \mathbf{H}(\mathbb{A}_{l,T})$ such that $\overline{\phi(\Gamma)}$ is compact and has non-empty interior. 
Then there are, uniquely determined,
\begin{itemize}
\item an injective map $w \colon V(l) \setminus T \to V(k) \setminus S$,
\item isomorphisms of valued fields $j_v \colon l_v \to k_{w(v)}$ for all $v \in V(l) \setminus T$ which induce an injection $j \colon \mathbb{A}_{l,T} \to \mathbb{A}_{k,S}$ of topological rings,
\item a homomorphism of group schemes over $j$
\[
	\eta\colon \mathbf{G}\times_k \mathbb{A}_{k,S} \to \mathbf{H}\times_l \mathbb{A}_{l,T},
\]
\item a group homomorphism $\nu\colon \Gamma \to Z(\mathbf{H})(\mathbb{A}_{l,T})$ 
\end{itemize}
such that $\phi(\gamma) = \nu(\gamma)\eta(\gamma)$ for all $\gamma \in \Gamma$.

Moreover, for every prime number $p$
\[
	\sum_{\substack{v \in V(l)\setminus T\\v\mid p}} [l_v:\mathbb{Q}_p] \leq \sum_{\substack{w \in V(k)\setminus S\\w\mid p}} [k_w:\mathbb{Q}_p] 
\]
and if equality occurs for every prime, then $w$ is a bijection and $j$ is an isomorphism.
\end{theorem}
The proof of Theorem 3.2 in \cite{KammeyerKionke} carries over almost verbally. One needs to verify that places not in $T$ are mapped to places not in $S$. To this end, it suffices to see that in \cite[Lemma 3.3]{KammeyerKionke}, the group $\overline{\phi_v(\Gamma)}$ is compact if and only if $\overline{\Gamma}$ is compact. 

The proof of Step 2 in \cite{KammeyerKionke} not correct as stated. We use the opportunity to give a correct argument. Assume that two places $v_1,v_2$ of $L$ are mapped to the same place $w$ of $K$.
Our aim (as in  \cite{KammeyerKionke}) is to show that the map $(\eta_{v_1},\eta_{v_2})\colon \mathbf{G}(K_w)\to \mathbf{H}(L_{v_1})\times\mathbf{H}(L_{v_2})$ has a nowhere dense image, to derive a contradiction. 
Since the image is a subgroup, it suffices to show that it does not contain an open neighbourhood of the identity. 
However, the inverse image of \(\mathbf{H}(L_{v_1})\times \{1\}\) under $(\eta_{v_1},\eta_{v_2})$ is the kernel of $\eta_{v_2}$.
We recall that $\eta_{v_1}$,$\eta_{v_2}$ are central isogenies, i.e.,  the kernel is finite and hence the intersection of the image $(\eta_{v_1},\eta_{v_2})$ with \(\mathbf{H}(L_{v_1})\times \{1\}\) is finite and cannot be open.

\end{document}